\documentclass[12pt]{amsart}
\usepackage{amssymb}
\usepackage{amsfonts}
\usepackage{latexsym}
\usepackage{amscd}
\usepackage{amsmath,delarray}
\usepackage[mathscr]{euscript}
\usepackage{xy} \xyoption{all}

\theoremstyle{plain}
\newtheorem{theorem}{Theorem}[section]
\newtheorem{lemma}[theorem]{Lemma}
\newtheorem{proposition}[theorem]{Proposition}
\newtheorem{corollary}[theorem]{Corollary}

\theoremstyle{definition}
\newtheorem{definition}[theorem]{Definition}
\newtheorem{definitions}[theorem]{Definitions}

\newtheorem{examples}[theorem]{Examples}

\newtheorem{remark}[theorem]{Remark}

\newtheorem*{remark*}{Remark}
\newtheorem*{remarks*}{Remarks}

\newtheorem*{assumption*}{Assumption}

\numberwithin{equation}{section}

\newcommand\ignore[1]{#1}

\newcommand{\dotedge}{\ar@{.}}
\newcommand{\eqedge}{\ar@{=}}

\newcommand{\N}{{\mathbb{N}}}
\newcommand{\Z}{{\mathbb{Z}}}
\newcommand{\Q}{{\mathbb{Q}}}

\newcommand{\Zplus}{{\mathbb{Z}^+}}

\newcommand{\calA}{\mathcal{A}}

\newcommand{\calD}{\mathcal{D}}

\newcommand{\calM}{\mathcal{M}}

\newcommand{\ol}{\overline}
\newcommand{\abar}{\ol{a}}
\newcommand{\fbar}{\ol{f}}
\newcommand{\qbar}{\ol{q}}
\newcommand{\xbar}{\ol{x}}
\newcommand{\ybar}{\ol{y}}
\newcommand{\zbar}{\ol{z}}

\newcommand{\Cbar}{\ol{C}}
\newcommand{\Ebar}{\ol{E}}
\newcommand{\Jbar}{\ol{J}}
\newcommand{\Mbar}{\ol{\calM}}

\newcommand{\psibar}{\ol{\psi}}

\newcommand{\Dtil}{\widetilde{D}}
\newcommand{\Etil}{\widetilde{E}}
\newcommand{\Ftil}{\widetilde{F}}

\newcommand{\eps}{\varepsilon}
\DeclareMathOperator{\soc}{soc}

\newcommand{\SSGr}{\mathbf{SSGr}}
\newcommand{\tri}{\vartriangleleft}

\DeclareMathOperator{\fin}{fin}
\newcommand{\Cfin}{C_{\fin}}
\DeclareMathOperator{\ped}{ped}

\begin{document}

\title{Tame and wild refinement monoids}

\author{P. Ara}

\address{Departament de Matem\`atiques, Universitat Aut\`onoma de Barcelona,
08193 Bellaterra (Barcelona), Spain.} \email{para@mat.uab.cat}

\author{K. R. Goodearl}

\address{Department of Mathematics, University of
California, Santa Barbara, CA 93106. }\email{goodearl@math.ucsb.edu}

\thanks{Part of this research was undertaken while the second-named author held a sabbatical fellowship from the Ministerio de Educaci\'on y Ciencias de Espa\~na at the Centre de Recerca Matem\`atica in Barcelona during spring 2011. 
He thanks both institutions for their support and hospitality. The first-named 
author was partially supported by DGI MINECO
MTM2011-28992-C02-01, by FEDER UNAB10-4E-378
``Una manera de hacer Europa", and by the
Comissionat per Universitats i
Recerca de la Generalitat de Catalunya.}

\begin{abstract} The class of refinement monoids (commutative monoids satisfying the Riesz refinement property) is subdivided into those which are tame, defined as being an inductive limit of 
finitely generated refinement monoids, and those which are wild, i.e., not tame. It is shown that tame refinement monoids enjoy many positive properties, including separative 
cancellation ($2x=2y=x+y \implies x=y$) and multiplicative cancellation with respect to the algebraic ordering ($mx\le my \implies x\le y$). In contrast, examples are constructed 
to exhibit refinement monoids which enjoy all the mentioned good properties but are nonetheless wild.
\end{abstract}

\maketitle

\section*{Introduction}

The class of \emph{refinement monoids} -- commutative monoids satisfying the Riesz refinement property -- has been extensively studied 
over the past few decades, in connection with the classification of countable Boolean algebras (e.g., \cite{Dob82, Ket, Pierce}) and the non-stable
K-theory of rings and C*-algebras (e.g., \cite{Areal, AGOP, AMP, GPW, PW06}), as well as for its own sake (e.g., \cite{Brook01, Dob, Dobb84, Gril76, W98}). 
Ketonen proved in \cite{Ket} that the set $BA$ of isomorphism classes of countable Boolean algebras, with the operation induced from direct products, 
is a refinement monoid, and that $BA$ contains all countable commutative monoid phenomena in that every countable commutative monoid embeds into $BA$. An important 
invariant in non-stable K-theory is the commutative monoid $V(R)$ associated to any ring $R$, consisting of the isomorphism classes of finitely generated projective (left, say) 
$R$-modules, with the operation induced from direct sum. If $R$ is a (von Neumann) regular ring or a C*-algebra with real rank zero (more generally, an exchange ring), 
then $V(R)$ is a refinement monoid (e.g., \cite[Corollary 1.3, Theorem 7.3]{AGOP}). The \emph{realization problem} asks which refinement monoids appear as a $V(R)$ for $R$ in 
one of the above-mentioned classes. Wehrung \cite{W98IsraelJ} constructed a conical refinement monoid of cardinality $\aleph_2$ which is not isomorphic to $V(R)$ for any regular 
ring $R$, but it is an open problem whether every countable conical refinement monoid can be realized as $V(R)$ for some regular $R$.

One observes that large classes of realizable refinement monoids satisfy many desirable properties that do not hold in general, in marked contrast to the ``universally bad" 
refinement monoid $BA$, which exhibits any property that can be found in a commutative monoid, such as elements $a$ and $b$ satisfying $2a=2b$ while $a\ne b$, 
or satisfying $a=a+2b$ while $a\ne a+b$. Moreover, the largest classes of realizable refinement monoids consist of inductive limits of simple ingredients, such as 
finite direct sums of copies of $\Zplus$ or $\{0,\infty\}$. These monoids are more universally realizable in the sense that they can be realized as $V(R)$ for regular 
algebras $R$ over any prescribed field. By contrast, examples are known of countable refinement monoids which are realizable only for regular algebras over some countable fields.
These examples are modelled on the celebrated construction of Chuang and Lee \cite{CL} (see \cite[Section 4]{Areal}).

These considerations lead us to separate the class of refinement monoids into subclasses of \emph{tame} and \emph{wild} refinement monoids, where the tame ones are the inductive
limits of finitely generated refinement monoids and the rest are wild. Existing inductive limit theorems allow us to identify several large classes of tame 
refinement monoids, such as unperforated cancellative refinement monoids; refinement monoids in which $2x=x$ for all $x$; and the \emph{graph monoids} introduced in \cite{AMP, AG}. We prove 
that tame refinement monoids satisfy a number of desirable properties, such as separative cancellation and lack of perforation (see \S\ref{background} for unexplained terms). Tame 
refinement monoids need not satisfy full cancellation, as $\{0,\infty\}$ already witnesses, but we show that among tame refinement monoids, stable finiteness ($x+y=x \implies y=0$) implies cancellativity. 

The collection of good properties enjoyed by tame refinement monoids known so far does not, as yet, characterize tameness. 
We construct two wild refinement monoids (one a quotient of the other) which are conical, stably finite, antisymmetric, separative, and unperforated; moreover, one of them is also 
archimedean. These monoids will feature in \cite{AGreal}, where an investigation of the subtleties of the realization problem will be carried out. In particular, we will show that one 
of the two monoids (the quotient monoid) is realizable by a von Neumann regular algebra over any field, but the other is realizable only by von Neumann regular algebras over {\it countable\/} fields.
We will also develop in \cite{AGreal} a connection with the Atiyah problem for the lamplighter group.

\section{Refinement monoids}

\subsection{Background and notation}  \label{background}

All monoids in this paper will be commutative, written additively, and homomorphisms between them will be assumed to be monoid homomorphisms. Categorical notions, such as inductive limits, will refer to the category of commutative monoids. The \emph{kernel} of a monoid homomorphism $\phi: A\rightarrow B$ is the congruence 
$$\ker(\phi) := \{ (a,a') \in A^2 \mid \phi(a) = \phi(a') \} .$$

We write $\Zplus$ for the additive monoid of nonnegative integers, and $\N$ for the additive semigroup of positive integers. The symbol $\sqcup $ stands for the disjoint union of sets.

A monoid $M$ is \emph{conical} if $0$ is the only invertible element of $M$, that is, elements $x,y \in M$ can satisfy $x+y=0$ only if $x=y=0$. Several levels of cancellation properties will be considered, as follows. First, $M$ is \emph{stably finite} if $x+y=x$ always implies $y=0$, for any $x,y \in M$. Second, $M$ is \emph{separative} provided $2x=2y=x+y$ always implies $x=y$, for any $x,y \in M$. There are a number of equivalent formulations of this property, as, for instance, in \cite[Lemma 2.1]{AGOP}. Further, $M$ is \emph{strongly separative} if $2x=x+y$ always implies $x=y$. Finally, $M$ is \emph{cancellative} if it satisfies full cancellation: $x+y=x+z$ always implies $y=z$, for any $x,y,z\in M$.

The \emph{algebraic ordering} (or \emph{minimal ordering}) in $M$ is the translation-invariant pre-order given by the existence of subtraction: elements $x,y \in M$ satisfy $x \le y$ if and only if there is some $z \in M$ such that $x+z=y$. If $M$ is conical and stably finite, this relation is a partial order on $M$. An \emph{order-unit} in $M$ is any element $u \in M$ such that all elements of $M$ are bounded above by multiples of $u$, that is, for any $x\in M$, there exists $m \in \Zplus$ such that $x\le mu$. The monoid $M$ is called \emph{unperforated} if $mx \le my$ always implies $x \le y$, for any $m \in \N$ and $x,y \in M$.

The monoid $M$ is said to be \emph{archimedean} if elements $x,y\in M$ satisfy $nx \le y$ for all $n\in\N$ only when $x$ is invertible. When $M$ is conical, this condition implies stable finiteness.

An \emph{o-ideal} of $M$ is any submonoid $J$ of $M$ which is hereditary with respect to the algebraic ordering, that is, whenever $x \in M$ and $y \in J$ with $x\le y$, it follows that $x\in J$. (Note that an o-ideal is not an ideal in the sense of semigroup theory.) The hereditary condition is equivalent to requiring that $x+z \in J$ always implies $x,z\in J$, for any $x,z\in M$. Given an o-ideal $J$ in $M$, we define the \emph{quotient monoid} $M/J$ to be the monoid $M/{\equiv}_J$, where ${\equiv}_J$ is the congruence on $M$ defined as follows: $x \equiv_J y$ if and only if there exist $a,b\in J$ such that $x+a = y+b$. Quotient monoids $M/J$ are always conical.

Separativity and unperforation pass from a monoid $M$ to any quotient $M/J$ \cite[Lemma 4.3]{AGOP}, but stable finiteness and the archimedean property do not, even in refinement monoids (Remark \ref{MMbarexamples}).

The monoid $M$ is called a \emph{refinement monoid} provided it satisfies the \emph{Riesz refinement property}: whenever $x_1,x_2,y_1,y_2 \in M$ with $x_1+x_2 = y_1+y_2$, there exist $z_{ij} \in M$ for $i,j=1,2$ with $x_i = z_{i1}+z_{i2}$ for $i=1,2$ and $y_j = z_{1j}+z_{2j}$ for $j = 1,2$. It can be convenient to record the last four equations in the format of a \emph{refinement matrix}
$$\bordermatrix{
 &y_1&y_2\cr  x_1&z_{11}&z_{12}\cr  x_2&z_{21}&z_{22}\cr},$$
where the notation indicates that the sum of each row of the matrix equals the element labelling that row, and similarly for column sums. By induction, analogous refinements hold for equalities between sums with more than two terms. A consequence of refinement is the \emph{Riesz decomposition property}: whenever $x,y_1,\dots,y_n \in M$ with $x \le y_1+ \cdots+ y_n$, there exist $x_1,\dots,x_n \in M$ such that $x = x_1+ \cdots+ x_n$ and $x_i \le y_i$ for all $i$.

The quotient of any refinement monoid by an o-ideal is a conical refinement monoid (e.g., \cite[p.~476]{Dob}, \cite[Proposition 7.8]{Brook97}).

The \emph{Riesz interpolation property} in $M$ is the following condition: Whenever $x_1,x_2,y_1,y_2 \in M$ with $x_i\le y_j$ for $i,j=1,2$, there exists $z\in M$ such that $x_i\le z\le y_j$ for $i,j=1,2$. If $M$ is cancellative and conical, so that it is the positive cone of a partially ordered abelian group, the Riesz refinement, decomposition, and interpolation properties are all equivalent (e.g., \cite[Proposition 2.1]{poagi}). In general, however, the only relation is that refinement implies decomposition.

For a refinement monoid, unperforation implies separativity. This follows immediately from \cite[Theorem 1]{Chen09}, and it was noted independently in  \cite[Corollary 2.4]{Wunpub}.

An element $p \in M$ is \emph{prime} if $p\le x+y$ always implies $p\le x$ or $p\le y$, for any $x,y\in M$. This follows the definition in \cite{Brook97, Brook01} as opposed to the one in \cite{APW}, which requires prime elements to be additionally non-invertible. The monoid $M$ is called \emph{primely generated} if every element of $M$ is a sum of prime elements. In case $M$ is conical, this is equivalent to the definition used in \cite{APW}. 

An element $x \in M$ is \emph{irreducible} if $x$ is not invertible and  $x=a+b$ only when $a$ or $b$ is invertible, for any $a,b\in M$. The following facts are likely known, but we did not locate any reference.

\begin{lemma}  \label{cancelirredelement}
In a conical refinement monoid $M$, all irreducible elements cancel from sums.
\end{lemma}

\begin{proof} Let $a,b,c\in M$ with $a+b=a+c$ and $a$ irreducible. There is a refinement
$$\bordermatrix{ &a&b\cr a&a'&b'\cr c&c'&d'}.$$
Since $a=a'+b'$, either $a'=0$ or $b'=0$. Likewise, either $a'=0$ or $c'=0$. If $a'=0$, we get $a=b'=c'$, and thus $b=a+d'=c$. If $a'\ne 0$, then $b'=c'=0$, and thus $b=d'=c$.
\end{proof}

\begin{proposition}  \label{irredelementgenideal}
Let $(a_i)_{i\in I}$ be a family of distinct irreducible elements in a conical refinement monoid $M$.

{\rm (a)} The submonoid $J := \sum_{i\in I} \Zplus a_i$ is an o-ideal of $M$.

{\rm (b)} The map $f:\bigoplus_{i\in I} \Zplus \rightarrow J$ sending $(m_i)_{i\in I} \mapsto \sum_{i\in I} m_ia_i$ is an isomorphism.
\end{proposition}

\begin{proof} (a) Let $b\in M$ and $c\in J$ with $b\le c$. If $c=0$, then $b=0$ because $M$ is conical, whence $b\in J$. Assume that $c \ne 0$, and write $c= \sum_{l=1}^n a_{i_l}$  for some $i_l\in I$. By Riesz decomposition, $b=  \sum_{l=1}^n b_l$ for some $b_l \in M$ with $b_l\le a_{i_l}$. By the irreducibility of the $a_{i_l}$, each $b_l= \eps_l a_{i_l}$ for some $\eps_l \in \{0,1\}$. Thus, $b= \sum_{l=1}^n \eps_l a_{i_l} \in J$.

(b) By definition, $f$ is surjective. To see that $f$ is injective, it suffices to prove the following:
\begin{enumerate}
\item[$(*)$] If $a_1,\dots,a_n$ are distinct irreducible elements in $M$ and $\sum_{i=1}^n m_ia_i= \sum_{i=1}^n m'_ia_i$ for some $m_i,m'_i\in \Zplus$, then $m_i=m'_i$ for all $i$.
\end{enumerate}
We proceed by induction on $t := \sum_{i=1}^n (m_i+m'_i)$. If $t=0$, then $m_i=0=m'_i$ for all $i$.

Now let $t>0$. Without loss of generality, $m'_1>0$. Hence, $a_1\le \sum_{i=1}^n m_ia_i$, so $a_1= \sum_{i=1}^n \sum_{l=1}^{m_i} b_{il}$ with each $b_{il} \le a_i$. Also, each $b_{il} \le a_1$. Some $b_{il} \ne 0$, whence $a_i=b_{il}= a_1$, yielding $i=1$ and $m_1>0$. Cancel $a_1$ from $\sum_{i=1}^n m_ia_i= \sum_{i=1}^n m'_ia_i$, leaving 
$$(m_1-1)a_1+ \sum_{i=2}^n m_ia_i= (m'_1-1)a_1+ \sum_{i=2}^n m'_ia_i.$$
By induction, $m_1-1= m'_1-1$ and $m_i=m'_i$ for all $i\ge2$, yielding $(*)$.
\end{proof}

\begin{definition}
 \label{def:socle}
 Let $M$ be a conical refinement monoid. Then the \emph{pedestal} of $M$, denoted by $\ped (M)$, is the submonoid of $M$ generated by all the irreducible elements of $M$.
 By Proposition \ref{irredelementgenideal}, $\ped (M)$ is an o-ideal of $M$.
 
With an eye on non-stable K-theory, it would seem reasonable to call the
submonoid defined above the socle of $M$.
This is due to the fact that if $R$ is a regular ring (or just a semiprime
exchange ring), then $V(\soc (R_R))\cong \ped (V(R))$,
where $\soc (R_R)$ is the socle of the right $R$-module $R_R$ in the sense of
module theory. However, the concept of the socle of a semigroup (e.g., \cite[Section 6.4]{CP}) is entirely different from our concept of a pedestal. The latter concept is designed for and works well 
in conical refinement monoids, but it may need modification for use with non-refinement monoids.
 \end{definition}

\section{Tame refinement monoids}  \label{tamerefmonoids}

\subsection{Tameness and wildness}  \label{tamewild}

\begin{definition}
 \label{def:wild-tamemonoids} Let $M$ be a refinement monoid. We say that $M$ is \emph{tame} in case $M$ is an inductive limit for some inductive system of finitely generated refinement monoids, and that $M$ is \emph{wild} otherwise.
  \end{definition}
  
\begin{examples}  \label{examples:tame}
Beyond finitely generated refinement monoids themselves, several classes of tame refinement monoids can be identified.

{\bf 1.} Every unperforated cancellative refinement monoid is tame. This follows from theorems of Grillet \cite{Gril} and Effros-Handelman-Shen \cite{EHS}. See Theorem \ref{tamecriterion:stablyfinite} below for details.

{\bf 2.} Any refinement monoid $M$ such that $2x=x$ for all $x\in M$ is tame. Recall that (upper) semilattices correspond exactly to semigroups satisfying the identity $2x=x$ for all elements $x$, where $\vee = +$. A semilattice is called \emph{distributive} if it satisfies Riesz decomposition \cite[p.~117]{Gr71}, and it is well known that this is equivalent to the semilattice satisfying Riesz refinement. Pudl\'ak proved in \cite[Fact 4, p.~100]{Pud} that every distributive semilattice equals the directed union of its finite distributive subsemilattices. In monoid terms, any refinement monoid $M$ satisfying $2x=x$ for all $x\in M$ is the directed union of those of its finite submonoids which satisfy refinement. Thus, such an $M$ is tame. In fact, any refinement monoid of this type is an inductive limit of finite boolean monoids (meaning semilattices of the form ${\mathbf 2}^n$ for $n\in\N$) \cite[Theorem 6.6, Corollary 6.7]{GW01}. 

{\bf 3.} Write $A\sqcup\{0\}$ for the monoid obtained from an abelian group $A$ by adjoining a new zero element. Any such monoid has refinement (e.g., \cite[Corollary 5]{Dobb84}). Inductive limits of monoids of the form $\bigoplus_{i=1}^k (A_i\sqcup\{0\})$ with each $A_i$ finite cyclic were characterized in \cite[Theorem 6.4]{GPW}, and those for which the $A_i$ can be arbitrary cyclic groups were characterized in \cite[Theorem 6.6]{PW06}. (We do not list the conditions here.) All these monoids are tame refinement monoids.

{\bf 4.} A commutative monoid $M$ is said to be {\it regular} in case $2x\le x$ for all $x\in M$.
It has been proved by Pardo and Wehrung \cite[Theorem 4.4]{PWunp} that every regular conical refinement monoid 
is a direct limit of finitely generated regular conical refinement monoids. In particular, these monoids are tame. 

{\bf 5.} A monoid $M(E)$ associated with any directed graph $E$ was defined for so-called row-finite graphs in \cite[p.~163]{AMP}, and then in general in \cite[p.~196]{AG}. These monoids have refinement by \cite[Proposition 4.4]{AMP} and \cite[Corollary 5.16]{AG}, and we prove below that they are tame (Theorem \ref{MEgeneral}).

{\bf 6.} In \cite[Theorem 0.1]{AP}, Pardo  and the first-named author proved that every primely generated conical refinement monoid is tame, thus resolving an open problem from the initial version of the present paper.
\end{examples}

The existence of wild refinement monoids follows, for instance, from the next theorem. Other examples will appear below. The following theorem and many other consequences of tameness rely on Brookfield's result  that every finitely generated refinement monoid is primely generated \cite[Corollary 6.8]{Brook01}, together with properties of primely generated refinement monoids established in \cite{Brook97, Brook01}.

\begin{theorem}  \label{tame:sepunperf}
Every tame refinement monoid is separative, unperforated, and satisfies the Riesz interpolation property.
\end{theorem}

\begin{proof} These properties obviously pass to inductive limits. For finitely generated refinement monoids, the first two properties were established in \cite[Theorem 4.5, Corollaries 5.11, 6.8]{Brook01}, and the third follows from \cite[Propositions 9.15 and 11.8]{Brook97}.
\end{proof}

By \cite[Theorem 1]{Gril} or \cite[Theorem 5.1]{Dob}, any monoid can be embedded in a refinement monoid. More strongly, for any monoid $M'$, there exists an embedding $\phi: M'\rightarrow M$ into a refinement monoid $M$ such that $\phi$ is also an embedding for the algebraic ordering \cite[p.~112]{W98}. If $M'$ is either perforated or not separative, then $M$ has the same property, and so by Theorem \ref{tame:sepunperf}, $M$ cannot be tame. Explicit examples of perforated refinement monoids (even cancellative ones) were constructed in \cite[Examples 11.17, 15.9]{Brook97}; these provide further wild monoids. Wild refinement monoids which are unperforated and separative also exist. One example is the monoid to which \cite{MDS} is devoted; others will be constructed below (see Theorems \ref{Mwild} and \ref{Mbarwild}).

\begin{remark}  \label{conicaltame}
A conical tame refinement monoid $M$ must be an inductive limit of conical finitely generated refinement monoids, as follows.  Write $M=\varinjlim _{i\in I} M_i$ for an inductive system of finitely generated refinement monoids $M_i$, with transition maps $\psi_{ij}: M_i \rightarrow M_j$ for $i\le j$ and limit maps $\psi^i: M_i \rightarrow M$. For each $i\in I$, the group of units $U(M_i)$ is an o-ideal of $M_i$, and the quotient $\ol{M}_i := M_i/U(M_i)$ is a conical finitely generated refinement monoid. The maps $\psi_{ij}$ and $\psi^i$ induce homomorphisms $\psibar_{ij} : \ol{M}_i \rightarrow \ol{M}_j$ and $\psibar^{\,i} : \ol{M}_i \rightarrow M$, and the monoids $\ol{M}_i$ together with the transition maps $\psibar_{ij}$ form an inductive system. It is routine to check that $M$ together with the maps $\psibar^{\,i}$ is an inductive limit for this system.
\end{remark}

\begin{proposition}
 \label{prop:tame-closed-under-quotients}
 Let $M$ be a tame refinement monoid and let $J$ be an o-ideal of $M$. Then both $J$ and $M/J$ are tame refinement monoids.
\end{proposition}

\begin{proof} Write $M=\varinjlim _{i\in I} M_i$ for an inductive system of finitely generated refinement monoids $M_i$, with transition maps $\psi_{ij}: M_i \rightarrow M_j$ for $i\le j$ and limit maps $\psi^i: M_i \rightarrow M$. For each $i\in I$, the set $J_i := (\psi^i)^{-1}(J)$ is an o-ideal of $M_i$, and the quotient $\ol{M}_i := M_i/J_i$ is a (conical) finitely generated refinement monoid. Observe that the maps $\psi_{ij}$ and $\psi^i$ induce homomorphisms $\psibar_{ij} : \ol{M}_i \rightarrow \ol{M}_j$ and $\psibar^{\,i} : \ol{M}_i \rightarrow M/J$, and the monoids $\ol{M}_i$ together with the transition maps $\psibar_{ij}$ form an inductive system. A routine check verifies that $M/J$ together with the maps $\psibar^{\,i}$ is an inductive limit for this system. Therefore $M/J$ is tame.

As is also routine, $J$ together with the restricted maps $\psi^i|_{J_i}$ is an inductive limit for the inductive system consisting of the $J_i$ and the maps $\psi_{ij}|_{J_i} : J_i \rightarrow J_j$. Each $J_i$ is a refinement monoid, and once we check that it is finitely generated, we will have shown that $J$ is tame. Thus, it just remains to verify the following fact:
\begin{enumerate}
\item[$\bullet$] If $N$ is a finitely generated monoid and $K$ an o-ideal of $N$, then $K$ is a finitely generated monoid.
\end{enumerate}
We may assume that $K$ is nonzero. Let $\{x_1,\dots,x_n\}$ be a finite set of generators for $N$. After permuting the indices, we may assume that the $x_i$ which lie in $K$ are exactly $x_1,\dots,x_m$, for some $m\le n$. Given $x \in K$,  write $x = a_1x_1+ \cdots+ a_nx_n$ for some $a_i \in \Zplus$. Whenever $a_i>0$, we have $x_i \le x$ and so $x_i \in K$, whence $i \le m$. Consequently, $x = a_1x_1+ \cdots+ a_mx_m$, proving that $K$ is generated by $x_1,\dots,x_m$. 
\end{proof}

\begin{theorem}  \label{tamecriterion}
A commutative  monoid $M$ is a tame refinement monoid if and only if
\begin{enumerate}
\item[$(\dagger)$] For each finitely generated submonoid $M' \subseteq M$, the inclusion map $M' \rightarrow M$ factors through a finitely generated refinement monoid.
\end{enumerate}
\end{theorem}

\begin{proof} By \cite[Lemma 4.1, Remark 4.3]{GPW}, $M$ is a direct limit of finitely generated refinement monoids (and thus a tame refinement monoid)  
if and only if the following conditions hold:
\begin{enumerate}
\item For each $x\in M$, there exist a finitely generated refinement monoid $N$ and a homomorphism $\phi: N\rightarrow M$ such that $x\in \phi(N)$.
\item For each finitely generated refinement monoid $N$, any homomorphism $\phi:N \rightarrow M$ equals the composition of homomorphisms $\psi:N \rightarrow N'$ and $\phi': N'\rightarrow M$ such that $N'$ is a finitely generated refinement monoid and $\ker\phi= \ker\psi$.
\end{enumerate}

Condition (1) is always satisfied, since for each $x\in M$, there is a homomorphism $\phi: \Zplus \rightarrow M$ such that $\phi(1)= x$.

Assume first that $(\dagger)$ holds, and let $\phi:N \rightarrow M$ be a homomorphism with $N$ a finitely generated refinement monoid. 
By $(\dagger)$, the inclusion map $\phi(N) \rightarrow M$ is a composition of homomorphisms $\theta: \phi(N) \rightarrow N'$ and $\phi': N'\rightarrow M$ with $N'$ a finitely generated refinement monoid. Then $\phi= \phi' \theta \phi_0$, where $\phi_0 : N \rightarrow \phi(N)$ is $\phi$ with codomain restricted to $\phi(N)$. Since $\theta$ is injective, it is clear that $\ker\phi= \ker\phi_0= \ker \theta\phi_0$.

Conversely, assume that  (2) holds, and let $M'$ be a finitely generated submonoid of $M$. Choose elements $x_1,\dots,x_n$ that generate $M'$, set $N := (\Zplus)^n$, and define $\phi: N\rightarrow M$ by the rule $\phi(m_1,\dots,m_n) = \sum_i m_ix_i$. This provides us with a finitely generated refinement monoid $N$ and a homomorphism $\phi: N\rightarrow M$ such that $M' = \phi(N)$. Write $\phi= \iota\phi_0$ where $\phi_0: N\rightarrow M'$ is $\phi$ with codomain restricted to $M'$ and $\iota: M' \rightarrow M$ is the inclusion map. Now let $\psi$, $\phi'$, and $N'$ be as in (2). Since $\ker\phi_0= \ker\phi= \ker\psi$, there is a unique homomorphism $\psi': M' \rightarrow N'$ such that $\psi'\phi_0 = \psi$. Then $\phi'\psi'\phi_0= \phi'\psi= \phi= \iota\phi_0$ and so $\phi'\psi'= \iota$. Thus $\iota$ factors through $N'$, as required.
\end{proof}

\begin{proposition}  \label{classoftame}
The class of tame refinement monoids is closed under direct sums, inductive limits, and retracts.
\end{proposition}

\begin{proof} Closure under inductive limits follows from \cite[Corollary 4.2, Remark 4.3]{GPW}, and then closure under retracts follows as in \cite[Lemma 4.4]{GPW}. Closure under finite direct sums is clear, and closure under arbitrary direct sums follows because such direct sums are inductive limits of finite direct sums, or by an application of Theorem \ref{tamecriterion}.
\end{proof}

\begin{definitions}  \label{primeandprimitive}
A monoid $M$ is \emph{antisymmetric} if the algebraic ordering on $M$ is antisymmetric (and thus is a partial order). In particular, antisymmetric monoids are conical, and any stably finite conical monoid is antisymmetric. If $M$ is an antisymmetric refinement monoid, its prime elements coincide with the \emph{pseudo-indecomposable} elements of \cite[p.~845]{Pierce}. A \emph{primitive monoid} \cite[Definition 3.4.1]{Pierce} is any antisymmetric, primely generated refinement monoid. Pierce characterized these monoids as follows.
\end{definitions}

\begin{proposition}  \label{primitivepresentation}
{\rm\cite[Proposition 3.5.2]{Pierce}}
The primitive monoids are exactly the commutative monoids with presentations of the form
\begin{equation}  \label{Dtriangle}
\langle D \mid e+f=f \;\, \text{for all} \;\, e,f\in D \;\, \text{with} \;\, e\tri f \rangle,
\end{equation}
where $D$ is a set and $\tri$ is a transitive, antisymmetric relation on $D$. If a monoid $M$ has such a presentation, then $D$ equals the set of nonzero prime elements of $M$, and elements $e,f\in D$ satisfy $e\tri f$ if and only if $e+f=f$.
\end{proposition}

Given any set $D$ equipped with a transitive, antisymmetric relation $\tri$, let us write $M(D,\tri)$ for the monoid with presentation \eqref{Dtriangle}.

Since antisymmetric monoids are conical, primitive monoids are tame by \cite[Theorem 0.1]{AP}. However, the result for primitive monoids is immediate from Pierce's results, as follows.

\begin{theorem}  \label{primitive=>tame}
Every primitive monoid is a tame refinement monoid.
\end{theorem}

\begin{proof} Let $M$ be a primitive monoid, with a presentation as in Proposition \ref{primitivepresentation}. If $D$ is finite, then $M$ is finitely generated and we are done, so assume that $D$ is infinite.

Let $\calD$ be the collection of nonempty finite subsets of $D$, partially ordered by inclusion. For $X\in \calD$, let $\tri_X$ denote the restriction of $\tri$ to $X$, and set $M_X := M(X,\tri_X)$. Since $\tri_X$ is transitive and antisymmetric, $M_X$ is primitive. In particular, $M_X$ is a refinement monoid, and it is finitely generated by construction. For any $X,Y \in \calD$ with $X\subseteq Y$, the inclusion map $X\rightarrow Y$ extends uniquely to a homomorphism $\psi_{X,Y}: M_X \rightarrow M_Y$. The collection of monoids $M_X$ and transition maps $\psi_{X,Y}$ forms an inductive system, and $M$ is an inductive limit for this system. Therefore $M$ is tame.
\end{proof}

\subsection{Further consequences of tameness}  \label{tameimplies}

Any monoid $M$ has a maximal antisymmetric quotient, namely $M/{\equiv}$ where $\equiv$ is the congruence defined as follows: $x \equiv y$ if and only if $x\le y\le x$. 

\begin{theorem}  \label{tame:antisymmetrizn}
If $M$ is a tame refinement monoid, then its maximal antisymmetric quotient $M/{\equiv}$ is a conical tame refinement monoid.
\end{theorem}

\begin{proof} Conicality for $M/{\equiv}$ follows immediately from antisymmetry.

Assume that $M$ is an inductive limit of an inductive system of finitely generated refinement monoids $M_i$. Then $M/{\equiv}$ is an inductive limit of 
the corresponding inductive system of finitely generated monoids $M_i/{\equiv}$. The latter monoids have refinement by \cite[Theorem 5.2 and Corollary 6.8]{Brook01}. Therefore $M$ has refinement and is tame.
\end{proof}

Moreira Dos Santos constructed an example in \cite{MDS} of a conical, unperforated, strongly separative refinement monoid whose maximal antisymmetric quotient does not have refinement. This monoid is wild, by Theorem \ref{tame:antisymmetrizn}.

\begin{theorem}  \label{tame:stablyfinite}
Let $M$ be a tame refinement monoid. If $M$ is stably finite, then $M$ is cancellative.
\end{theorem}

\begin{proof} Suppose $a,b,c\in M$ with $a+c= b+c$, and let $M'$ be the submonoid of $M$ generated by $a$, $b$, $c$. By Theorem \ref{tamecriterion}, the inclusion map $M' \rightarrow M$ can be factored as the composition of homomorphisms $f: M'\rightarrow M''$ and $g: M''\rightarrow M$ where $M''$ is a finitely generated refinement monoid. Apply \cite[Corollary 6.8]{Brook01} and \cite[Lemma 2.1]{APW} to the equation $f(a)+f(c) = f(b)+f(c)$ in $M''$. This yields elements $x,y\in M''$ such that $f(a)+x = f(b)+y$ and $f(c)+x= f(c)+y= f(c)$. Then $c+g(x)= c+g(y) =c$, and so $g(x)=g(y)=0$ by the stable finiteness of $M$. Therefore $a= a+g(x)= b+g(y) = b$, as desired.
\end{proof}

We shall construct examples of stably finite, noncancellative refinement monoids below. These will be wild by Theorem \ref{tame:stablyfinite}.

\begin{corollary} \label{tamecancellativequo}
Let $M$ be a tame refinement monoid, and set
$$J := \{a \in M \mid \text{there exists} \;\, b \in M \; \text{with} \;\, a+b\le b \}.$$
Then $J$ is an o-ideal of $M$, and $M/J$ is a cancellative tame refinement monoid.
\end{corollary}

\begin{proof} It is clear that $J$ is an o-ideal. The quotient $M/J$ thus exists, and it is a tame refinement monoid by Proposition \ref{prop:tame-closed-under-quotients}. By Theorem \ref{tame:stablyfinite}, it only remains to show that $M/J$ is stably finite.

Suppose $x$ and $a$ are elements of $M$ whose images $\xbar,\abar\in M/J$ satisfy $\xbar+\abar= \xbar$. Then there exist $u,v\in J$ such that $x+a+u = x+v$. By definition of $J$, there is some $b\in M$ such that $v+b\le b$. Hence, 
$$a+(x+b) \le x+a+u+b = x+v+b \le x+b.$$
Therefore $a\in J$ and $\abar=0$, as required.
\end{proof}

Among stably finite conical refinement monoids, tameness can be characterized by combining Theorem \ref{tame:stablyfinite} with criteria of Grillet \cite{Gril76} and Effros-Handelman-Shen \cite{EHS}, as follows.

\begin{theorem}  \label{tamecriterion:stablyfinite}
Let $M$ be a stably finite conical refinement monoid. Then $M$ is tame if and only if it is unperforated and cancellative.
\end{theorem}

\begin{proof} Necessity is given by Theorems \ref{tame:sepunperf} and \ref{tame:stablyfinite}. Conversely, assume that $M$ is unperforated and cancellative. Since $M$ is also conical, its algebraic order is antisymmetric. We now apply one of two major theorems.

The first approach stays within commutative monoids. From the above properties of $M$, we see that the following conditions hold for any $a,c,d\in M$ and $n\in\N$: (1) If $na=nc$, then $a=c$, and (2) if $na= nc+d$, then the element $d$ is divisible by $n$ in $M$. By \cite[Proposition 2.7]{Gril76}, $M$ satisfies what is there called the \emph{strong RIP}: Whenever $a,b,c,d\in M$ and $n\in \N$ with $na+b= nc+d$, there exist $u,v,w,z\in M$ such that 
$$a=u+v,\qquad b=nw+z,\qquad c=u+w,\qquad d=nv+z.$$
The main theorem of \cite{Gril76}, Theorem 2.1, now implies that $M$ is an inductive limit of an inductive system of free commutative monoids, i.e., of direct sums of copies of $\Zplus$. Clearly free commutative monoids are tame, and therefore $M$ is tame by Proposition \ref{classoftame}.

For the second approach, observe that $M$ is the positive cone of a directed, partially ordered abelian group $G$ (because $M$ is conical and cancellative). The assumptions on $M$ imply that $G$ is an unperforated interpolation group (with respect to its given partial order), and thus $G$ is a dimension group. The theorem of Effros-Handelman-Shen \cite[Theorem 2.2]{EHS} now implies that $G$ is an inductive limit of an inductive system of partially ordered abelian groups $\Z^{n_i}$. Consequently, $M$ is an inductive limit of an inductive system of monoids $(\Zplus)^{n_i}$, and therefore $M$ is tame.
\end{proof}

Tame refinement monoids satisfy a number of other properties, of which we mention a few samples.

\begin{remark}  \label{tame:further}
Let $M$ be a tame refinement monoid and $a,b,c,d_1,d_2 \in M$. Then:
\begin{enumerate}
\item If $a+c \le b+c$, there exists $a_1\in M$ such that $a_1+c=c$ and $a\le b+a_1$.
\item If $a\le c+d_i$ for $i=1,2$, there exists $d\in M$ such that $a\le c+d$ and $d\le d_i$ for $i=1,2$.
\end{enumerate}
Both (1) and (2) hold in finitely generated refinement monoids, by \cite[Corollaries 4.2, 5.17, 6.8]{Brook01}, and they pass to inductive limits.
\end{remark}

We close this section with a result due to the referee. We thank him/her for allowing us to include it here.

\begin{proposition}
 \label{prop:Riesz-prop}
 Let $M$ be a tame refinement monoid, and let $\sim $ be the congruence on $M$ given by $x\sim y $ if and only if there exists $z \in M$ such that $x+z=y+z$. 
 Then $M/{\sim} $ {\rm(}the maximal cancellative quotient of $M$\/{\rm)} is a Riesz monoid. 
\end{proposition}

\begin{proof}
 Denote by $[x]$ the $\sim$-equivalence class of $x\in M$.
 
 Suppose that $[a]\le [b_1]+[b_2]$ in $M/{\sim} $. There is $c\in M$ such that $a+c\le b_1+b_2+c$, and so, by Remark \ref{tame:further}(1), there is $d\in M$ such that 
$d+c= c$ and  $a\le b_1+b_2+d$. By Riesz decomposition in $M$, we have $a=a_1+a_2+e$, where $a_i\le b_i$, $i=1,2$, and $e\le d$. Now, $a_1\le b_1$ and $e+c\le c$, whence $a_1+e+c\le b_1+c$, and thus $[a_1+e]\le [b_1]$. 
Since $[a_2]\le [b_2]$ and $[a]=[a_1+e]+[a_2]$, we have thus verified that $M/{\sim} $ has Riesz decomposition.
\end{proof}

We believe that the first example of a regular ring $R$ such that $V(R)$ is a wild refinement monoid is Bergman's example \cite[Example 5.10]{vnrr}.
This ring is stably finite but not unit-regular, so $V(R)$ is wild by Theorem \ref{tame:stablyfinite}. 
As noted in \cite[Section 3]{AGreal}, the regular rings constructed by Bergman in \cite[Example 5.10]{vnrr} and by Menal and Moncasi in \cite[Example 2]{MM} realize 
the monoid $\Mbar$ constructed in Section \ref{MMbar}, so that $\Mbar$ could be considered as the most elementary example of a wild refinement monoid.
Moncasi constructed in \cite{moncasi} an example of a regular ring $R$ such that $K_0(R)$ is not a Riesz group. Hence, $K_0(R)^+$ is not a Riesz monoid. But $K_0(R)^+ \cong V(R)/{\sim}$. Thus, by 
Proposition \ref{prop:Riesz-prop}, $V(R)$ is a wild refinement monoid.

\section{Graph monoids}
\label{sec:graphmon}

We express directed graphs in the form $E := (E^0, E^1, r, s)$ where $E^0$ and $E^1$ are the sets of vertices and arrows of $E$, respectively, while $r = r_E$ and $s = s_E$ are the respective range and source maps $E^1 \rightarrow E^0$. The graph $E$ is said to be \emph{row-finite} if its incidence matrix is row-finite, i.e., for each vertex $v \in E^0$, there are at most finitely many arrows in $E^1$ with source $v$. There is a natural \emph{category of directed graphs}, call it $\calD$, whose objects are all directed graphs and in which a morphism from an object $E$ to an object $F$ is any pair of maps $(g_0,g_1)$ where $g_i: E^i \rightarrow F^i$ for $i=1,2$ while $r_F g_1 = g_0 r_E$ and $s_F g_1 = g_0 s_E$. Any inductive system in $\calD$ has an inductive limit in $\calD$.

\subsection{Monoids associated to (unseparated) directed graphs}  \label{M(E)}

A \emph{graph monoid} $M(E)$ associated to a directed graph $E$ was first introduced in \cite[p.~163]{AMP} in the case that $E$ is row-finite. In that setting, $M(E)$ is defined to be the commutative monoid presented by the set of generators $E^0$ and the relations
\begin{enumerate}
\item $v = \sum \{ r(e) \mid  e \in s^{-1}(v) \}$ for all non-sinks $v \in E^0$.
\end{enumerate}
A definition for $M(E)$ in the general case was given in \cite[p.~196]{AG}. Then, the generators $v\in E^0$ are supplemented by generators $q_Z$ as $Z$ runs through all nonempty finite subsets of $s^{-1}(v)$ for \emph{infinite emitters} $v$, that is, vertices $v\in E^0$ such that $s^{-1}(v)$ is infinite. The relations consist of
\begin{enumerate}
\item $v = \sum \{ r(e) \mid  e \in s^{-1}(v) \}$ for all $v \in E^0$ such that $s^{-1}(v)$ is nonempty and finite.
\item $v = \sum \{r(e)\mid e\in Z\} +q_Z$ for all infinite emitters $v \in E^0$ and all nonempty finite subsets $Z\subset s^{-1}(v)$.
\item $q_Z = \sum \{r(e)\mid e\in W \setminus Z\} +q_W$ for all nonempty finite sets $Z \subseteq W \subset s^{-1}(v)$, where $v\in E^0$ is an infinite emitter.
\end{enumerate}

We give two proofs that the graph monoids $M(E)$ are tame refinement monoids. The first (Theorem \ref{MEgeneral}) involves changing the graph $E$ to a new graph $\Etil$ whose monoid is 
isomorphic to $M(E)$, while the second (see \S\ref{sepgraphmonoids}) takes advantage of inductive limit results for monoids associated to separated graphs (Definition \ref{defsepgraph}).
Observe that both proofs give that $M(E)$ is in fact an inductive limit of {\it graph monoids} $M(F)$ associated to {\it finite graphs} $F$, showing thus an {\it a priori} stronger statement.   

\begin{theorem}  \label{MEgeneral}
For any directed graph $E$, the graph monoid $M(E)$ is a tame, conical, refinement monoid. In particular, it is unperforated and separative.
\end{theorem}

\begin{proof} It is known that $M(E)$ is always a conical refinement monoid. In the row-finite case, conicality is easily seen from the definition of $M(E)$, or one can obtain it from \cite[Theorem 3.5]{AMP}, and refinement was proved in \cite[Proposition 4.4]{AMP}. These two properties of $M(E)$ were proved in general in \cite[Corollary 5.16]{AG}. 

Tameness in the row-finite case follows from \cite[Lemma 3.4]{AMP}, in which it was proved that $M(E)$ is an inductive limit of graph monoids $M(E_i)$, for certain finite subgraphs $E_i$ of $E$. Since each $M(E_i)$ is a finitely generated refinement monoid, tameness follows. We will show that tameness in general can be reduced to the ``row-countable" case, and that the latter case follows from the row-finite case.

Let $\calA$ denote the collection of those subsets $A\subseteq E^1$ such that
\begin{enumerate}
\item $s^{-1}(v)\subseteq A$ for all $v\in E^0$ such that $s^{-1}(v)$ is finite.
\item $s^{-1}(v)\cap A$ is countably infinite for all infinite emitters $v\in E^0$.
\end{enumerate}
For each $A\in \calA$, let $E_A$ denote the subgraph $(E^0,A,r|_A,s|_A)$ of $E$. Since $\calA$ is closed under finite unions, the graphs $E_A$ together with the inclusion maps $E_A \rightarrow E_B$ for $A\subseteq B$ in $\calA$ form an inductive system, and $E$ is an inductive limit of this system in the category $\calD$. While the functor $M(-)$ does not preserve all inductive limits, it does preserve ones of the form just described, as is easily verified. Consequently, to prove that $M(E)$ is tame, it suffices to prove that each $M(E_A)$ is tame, by Proposition \ref{classoftame}.

It remains to deal with the case in which $s^{-1}(v)$ is countable for all $v \in E^0$. Let $E^0_\infty$ denote the set of all infinite emitters in $E^0$, and for each $v \in E^0_\infty$, list the arrows emitted by $v$ as a sequence without repetitions, say
$$s^{-1}(v) = \{ e_{v,1}, e_{v,2}, \dots\}.$$
Then set $q_{v,n} := q_{\{e_{v,1}, \dots, e_{v,n}\}}$ for $v \in E^0_\infty$ and $n\in \N$, and observe that $M(E)$ has a presentation with generating set
$$E^0 \sqcup \{ q_{v,n} \mid v \in E^0_\infty, \; n \in \N\}$$
and relations
\begin{align*}
v &= \sum \{ r(e) \mid  e \in s^{-1}(v) \} \,, &&\quad (v\in E^0,\; 0< |s^{-1}(v)|< \infty)  \\
v &= r(e_{v,1}) +q_{v,1} \,, &&\quad (v \in E^0_\infty)  \\
q_{v,n} &= r(e_{v,n+1})+ q_{v,n+1} \,, &&\quad  (v \in E^0_\infty, \; n \in \N).
\end{align*}

Define a new directed graph $\Etil$ with vertex set 
$$\Etil^0 := E^0 \sqcup \{ w_{v,n} \mid v \in E^0_\infty, \; n \in \N \}$$
and $\Etil^1$ consisting of the following arrows:
\begin{enumerate}
\item $e \in E^1$, for those $e$ such that $s(e) \notin E^0_\infty$;
\item $e_{v,1}$ and an arrow $v \rightarrow w_{v,1}$, for each $v \in E^0_\infty$;
\item Arrows $w_{v,n} \rightarrow w_{v,n+1}$ and $w_{v,n} \rightarrow r(e_{v,n+1})$,  for each $v \in E^0_\infty$ and $n\in \N$.
\end{enumerate}
This graph is row-finite, and so $M(\Etil)$ is tame. A comparison of presentations shows that there is an isomorphism $M(E) \rightarrow M(\Etil)$, sending $v\mapsto v$ for all $v\in E^0$ and $q_{v,n} \mapsto w_{v,n}$  for all $v \in E^0_\infty$ and $n\in \N$. Therefore $M(E)$ is tame, as desired.
\end{proof}

Tame conical refinement monoids, even finitely generated antisymmetric ones, do not all arise as graph monoids $M(E)$, as shown in \cite[Lemma 4.1]{APW}.

\subsection{Separated graphs and their monoids}  \label{sepgraphmonoids}

\begin{definition}  \label{defsepgraph}
A \emph{separated graph}, as defined in \cite[Definition 2.1]{AG},  is a pair $(E,C)$ where $E$ is a directed graph,  $C=\bigsqcup
_{v\in E^0} C_v$, and
$C_v$ is a partition of $s^{-1}(v)$ (into pairwise disjoint nonempty
subsets) for every vertex $v$. (In case $v$ is a sink, we take $C_v$
to be the empty family of subsets of $s^{-1}(v)$.) The pair $(E,C)$ is called \emph{finitely separated} provided all the members of $C$ are finite sets.

We first give the definition of the \emph{graph monoid} $M(E,C)$ associated to a finitely separated graph $(E,C)$. This is the commutative monoid presented by the set of generators $E^0$ and the relations
$$v = \sum_{e \in X} r(e) \qquad \text{for all} \; v \in E^0 \; \text{and} \; X \in C_v \,.$$
\end{definition}

Lemma 4.2 of \cite{AG} shows that $M(E,C)$ is conical, and that it is nonzero as long as $E^0$ is nonempty. 
Otherwise, $M(E,C)$ has no special properties, in contrast to Theorem \ref{MEgeneral}:

\begin{theorem}  \label{MECarb}
{\rm\cite[Proposition 4.4]{AG}} Any conical commutative monoid is isomorphic to $M(E,C)$ for a suitable finitely separated graph $(E,C)$.
\end{theorem}

\begin{definition}  \label{defMEC}
The general definition of graph monoids in \cite[Definition 4.1]{AG} applies to triples $(E,C,S)$, where $(E,C)$ is a separated graph and $S$ is any subset of the collection
$$\Cfin := \{ X\in C \mid X \;\, \text{is finite} \}.$$
These triples are the objects of a category $\SSGr$ \cite[Definition 3.1]{AG} in which the morphisms $\phi: (E,C,S) \rightarrow (F,D,T)$ are those graph morphisms $\phi= (\phi^0,\phi^1): E\rightarrow F$ such that
\begin{enumerate}
\item[$\bullet$] $\phi^0$ is injective.
\item[$\bullet$] For each $X\in C$, the restriction of $\phi^1$ to $X$ is an injection of $X$ into some member of $D$.
\item[$\bullet$] For each $X\in S$, the restriction $\phi^1|_X$ is a bijection of $X$ onto a member of $T$.
\end{enumerate}

For any object $(E,C,S)$ of $\SSGr$, the \emph{graph monoid} $M(E,C,S)$ is presented by the generating set
$$E^0 \sqcup \{ q'_Z \mid Z \subseteq X \in C, \;\, 0 < |Z|< \infty\}$$
and the relations
\begin{enumerate}
\item[$\bullet$] $v= q'_Z + \sum_{e\in Z} r(e)$ for all $v\in E^0$ and nonempty finite subsets $Z$ of members of $C_v$.
\item[$\bullet$] $q'_{Z_1} = q'_{Z_2}+ \sum_{e\in Z_2\setminus Z_1} r(e)$ for all nonempty finite subsets $Z_1\subseteq Z_2$ of members of $C$.
\item[$\bullet$] $q'_X=0$ for all $X\in S$.
\end{enumerate}
In case $S = \Cfin$, we set $M(E,C) := M(E,C,\Cfin)$. When $(E,C)$ is finitely separated, the generators $q'_Z$ are redundant, and $M(E,C)$ has the presentation described in Definition \ref{defsepgraph}.

Finally, we define $M(E) := M(E,C)$ where $C$ is the ``unseparation" of $E$, namely, the collection of all sets $s^{-1}(v)$ for non-sinks $v\in E^0$. 
\end{definition}

We note that whenever $(E,C)$ is a separated graph and there is a (directed) path from a vertex $v$ to a vertex $w$ in $E$, we have $w \le v$ with respect to the algebraic ordering in any $M(E,C,S)$. It is enough to verify this statement when there is an arrow $e : v \rightarrow w$. In that case, $w\le v$ follows from the relation $v= q'_{\{e\}}+ w$.
\medskip

\noindent\emph{Alternative proof of Theorem {\rm\ref{MEgeneral}}}. As before, $M(E)$ is a conical refinement monoid by the results of \cite{AMP, AG}. 

As above, write $M(E)= M(E,C,S)$ where $C := \{ s^{-1}(v) \mid v\in E^0, \;\, |s^{-1}(v)| > 0\}$ and $S := \Cfin$. By \cite[Proposition 3.5]{AG}, $(E,C,S)$ is an inductive limit in $\SSGr$ of its finite complete subobjects $(F,D,T)$, where finiteness means that $F^0$ and $F^1$ are finite sets while completeness in this case just means that
\begin{enumerate}
\item[$\bullet$] $D = \{ s_F^{-1}(v) \mid v\in F^0, \;\, |s_F^{-1}(v)| > 0\}$, and
\item[$\bullet$] $T = \{ s_E^{-1}(v) \mid v\in F^0, \;\, 0< |s_E^{-1}(v)| <\infty\}$.
\end{enumerate} 
As noted in \cite[p.~186]{AG}, the functor $M(-)$ from $\SSGr$ to commutative monoids preserves inductive limits, so $M(E)$ is an inductive limit of the monoids $M(F,D,T)$. These monoids are finitely generated because the graphs $F$ are finite, so we just need them to have refinement.

Each $(F,D)$ is a finitely separated graph, but $M(F,D,T)$ may not be equal to $M(F)$, because $T$ is in general a proper subset of $D_{\text{fin}}$. We apply \cite[Construction 5.3]{AG} to $(F,D,T)$, to obtain a finitely separated graph $(\Ftil,\Dtil)$ such that $M(F,D,T) \cong M(\Ftil,\Dtil)$. It is clear from the construction that $\Dtil_v = s^{-1}_{\Ftil}(v)$ for all non-sinks $v\in \Dtil^0$, and so $M(\Ftil,\Dtil)= M(\Ftil)$. Since $M(\Ftil)$ is a refinement monoid \cite[Proposition 4.4]{AMP}, so is $M(F,D,T)$, as required.
\qed

\section{Two wild examples}  \label{MMbar}

We present explicit constructions of two refinement monoids which are wild, but otherwise possess most of the properties of tame refinement monoids established above. In particular, they are stably finite, separative, unperforated, and conical, and one of them is archimedean.
Both are constructed as graph monoids of separated graphs.

\subsection{Three separated graphs}  \label{3sepgraphs}

We concentrate on three particular separated graphs, denoted $(E_0,C^0)$, $(E,C)$, and $(\Ebar,\Cbar)$, which are drawn in Figures \ref{E0C0} and \ref{EC} below. Both of the latter two graphs contain $E_0$ as a subgraph. We indicate the sets in the families $C^0$, $C$, and $\Cbar$  by connecting their members with dotted lines. 
Thus, $C^0_u= \bigl\{ \{e_1,e_2\}, \{f_1,f_2\} \bigr\}$, while $C^0_{x_0}$, $C^0_{y_0}$, $C^0_{z_0}$ are empty.
\ignore{
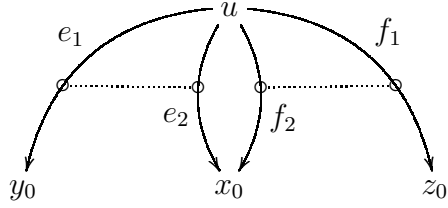
\begin{figure}[htp]
$$
\xymatrixrowsep{4.5pc}\xymatrixcolsep{5pc}\def\labelstyle{\displaystyle}
\xymatrix{
 &u  \ar@/_4ex/[dl]_{e_1}  \ar@/_4ex/[dl]|(0.65){\circ}="e1"  \ar@/_1pc/[d]_(0.6){e_2}  \ar@/_1pc/[d]|(0.45){\circ}="e2"  \ar@/^1pc/[d]^(0.6){f_2}  \ar@/^1pc/[d]|(0.45){\circ}="f2"  \ar@/^4ex/[dr]^{f_1}  \ar@/^4ex/[dr]|(0.65){\circ}="f1"    \dotedge"e1";"e2"  \dotedge"f1";"f2"  \\
y_0 &x_0 &z_0
}$$
\caption{The separated graph $(E_0,C^0)$}   \label{E0C0}
\end{figure}
}

As we shall see below, the graph monoid of $(E_0,C^0)$ does not have refinement (Remark \ref{M0notref}). In \cite[Construction 8.8]{AG}, a process of \emph{complete resolutions} was developed, by which any finitely separated graph can be enlarged to one whose graph monoid has refinement \cite[Theorem 8.9]{AG}. One application of this process leads to $(E,C)$ (we leave the details to the reader). Since $|C_v|\le 2$ for all $v\in E^0$, $(E,C)$ is also a \emph{complete multiresolution} of $(E_0,C^0)$ in the sense of \cite[Section 3]{AE}.  Finally, $(\Ebar,\Cbar)$ is obtained by removing the vertices $a_1,a_2,\dots$ from $E^0$ and shrinking the sets in $C$ as indicated in the diagram.
\ignore{
\begin{figure}[htp]
$$\xymatrixrowsep{1.33pc}\xymatrixcolsep{4pc}\def\labelstyle{\displaystyle}
\xymatrix{
 &u  \ar@/_3ex/[dddl]  \ar@{}@/_3ex/[dddl]|(0.65){\circ}="e1"  \ar@/_2ex/[ddd]  \ar@{}@/_2ex/[ddd]|{\circ}="e2"  \ar@/^2ex/[ddd]  \ar@{}@/^2ex/[ddd]|{\circ}="f2"  \ar@/^3ex/[dddr]  \ar@{}@/^3ex/[dddr]|(0.65){\circ}="f1"  \dotedge"e1";"e2"  \dotedge"f1";"f2" 
 &&&u  \ar@/_3ex/[dddl]  \ar@{}@/_3ex/[dddl]|(0.65){\circ}="ee1"  \ar@/_2ex/[ddd]  \ar@{}@/_2ex/[ddd]|{\circ}="ee2"  \ar@/^2ex/[ddd]  \ar@{}@/^2ex/[ddd]|{\circ}="ff2"  \ar@/^3ex/[dddr]  \ar@{}@/^3ex/[dddr]|(0.65){\circ}="ff1"  \dotedge"ee1";"ee2"  \dotedge"ff1";"ff2"  \\  \\  \\
y_0  \ar[ddd]  \ar[ddd]|(0.5){\circ}="g12"  \ar@/^2ex/[ddr]  \ar@/^2ex/[ddr]|(0.4){\circ}="g11"  
&x_0  \ar[dddl]  \ar[dddl]|(0.65){\circ}="h12"  \ar@/_4ex/[ddd]  \ar@/_4ex/[ddd]|(0.65){\circ}="h22"  \ar@/^4ex/[ddd]  \ar@/^4ex/[ddd]|(0.65){\circ}="g22"  \ar[dddr]  \ar[dddr]|(0.65){\circ}="g21" &z_0  \ar@/_2ex/[ddl]  \ar@/_2ex/[ddl]|(0.4){\circ}="h11"  \ar[ddd]  \ar[ddd]|(0.5){\circ}="h21"  \dotedge"g12";"g11" \dotedge"g21";"g22" \dotedge"h11";"h21" \dotedge"h12";"h22"  
 &y_0  \ar[ddd]  \ar[ddd]|(0.5){\circ}  
&x_0  \ar[dddl]  \ar[dddl]|(0.65){\circ}="hh12"  \ar@/_4ex/[ddd]  \ar@/_4ex/[ddd]|(0.65){\circ}="hh22"  \ar@/^4ex/[ddd]  \ar@/^4ex/[ddd]|(0.65){\circ}="gg22"  \ar[dddr]  \ar[dddr]|(0.65){\circ}="gg21" &z_0   \ar[ddd]  \ar[ddd]|(0.5){\circ}  \dotedge"gg21";"gg22"  \dotedge"hh12";"hh22"  \\  \\
 &a_1  \\
y_1 \ar@{}[dd]|{\vdots}  &x_1 \ar@{}[dd]|{\vdots} &z_1 \ar@{}[dd]|{\vdots}
 &y_1 \ar@{}[dd]|{\vdots}  &x_1 \ar@{}[dd]|{\vdots} &z_1 \ar@{}[dd]|{\vdots}  \\  \\
 y_{n-1}  \ar[ddd]  \ar[ddd]|(0.5){\circ}="g12"  \ar@/^2ex/[ddr]  \ar@/^2ex/[ddr]|(0.4){\circ}="g11"  &x_{n-1}  \ar[dddl]  \ar[dddl]|(0.65){\circ}="h12" \ar@/_4ex/[ddd] \ar@/_4ex/[ddd]|(0.65){\circ}="h22"  \ar@/^4ex/[ddd]  \ar@/^4ex/[ddd]|(0.65){\circ}="g22"  \ar[dddr]  \ar[dddr]|(0.65){\circ}="g21" &z_{n-1}  \ar@/_2ex/[ddl]  \ar@/_2ex/[ddl]|(0.4){\circ}="h11"  \ar[ddd]  \ar[ddd]|(0.5){\circ}="h21"  \dotedge"g12";"g11" \dotedge"g21";"g22" \dotedge"h11";"h21" \dotedge"h12";"h22" 
 &y_{n-1}  \ar[ddd]  \ar[ddd]|(0.5){\circ}
  &x_{n-1}  \ar[dddl]  \ar[dddl]|(0.65){\circ}="hh12" \ar@/_4ex/[ddd] \ar@/_4ex/[ddd]|(0.65){\circ}="hh22"  \ar@/^4ex/[ddd]  \ar@/^4ex/[ddd]|(0.65){\circ}="gg22"  \ar[dddr]  \ar[dddr]|(0.65){\circ}="gg21" &z_{n-1}   \ar[ddd]  \ar[ddd]|(0.5){\circ}  \dotedge"gg21";"gg22"  \dotedge"hh12";"hh22"  \\  \\
 &a_n  \\
y_n \ar@{}[dd]|{\vdots} &x_n \ar@{}[dd]|{\vdots} &z_n \ar@{}[dd]|{\vdots}
 &y_n \ar@{}[dd]|{\vdots} &x_n \ar@{}[dd]|{\vdots} &z_n \ar@{}[dd]|{\vdots}  \\  \\
 &{\save+<0ex,-2ex> \drop{(E,C)} \restore} & &&{\save+<0ex,-2ex> \drop{(\Ebar,\Cbar)} \restore} &
}$$
\caption{}  \label{EC}
\end{figure}
}

\subsection{Three graph monoids}  \label{3graphmons}

We label the monoids of the three separated graphs introduced in \S\ref{3sepgraphs} as follows:
$$\calM_0:= M(E_0,C^0), \qquad\quad \calM := M(E,C), \qquad\quad \Mbar := M(\Ebar,\Cbar).$$
These are conical commutative monoids, as noted above. Note that $u$ is an order-unit in each of these monoids, because there are paths from $u$ to each vertex of any of the three graphs.  In all three monoids, we have $u= x_0+y_0$, which means that $u$ can be omitted from the set of generators. In particular, $\calM_0$ has the monoid presentation
$$\calM_0 = \langle x_0, \, y_0, \, z_0 \mid x_0+y_0 = x_0+z_0 \rangle.$$

We can present $\calM$ by the generators
$$x_0,y_0,z_0,a_1,x_1,y_1,z_1,a_2,x_2,y_2,z_2,\dots$$
and the relations
\begin{equation}
\begin{gathered}
x_0+y_0 = x_0+z_0\,,\qquad\qquad  y_l = y_{l+1}+a_{l+1}\,,\qquad\qquad  z_l = z_{l+1}+a_{l+1}\,,  \\  
x_l = x_{l+1}+y_{l+1} =x_{l+1}+z_{l+1}\,.
\end{gathered}  \label{MECrelns}
\end{equation}
By \cite[Proposition 5.9 or Theorem 8.9]{AG}, $\calM$ is a refinement monoid. We give a direct proof of this in Proposition \ref{Mref}.

The monoid $\Mbar$ can be presented by the generators
$$x_0,y_0,z_0,x_1,y_1,z_1,x_2,y_2,z_2,\dots$$
and the relations
\begin{equation}
\begin{gathered}
x_0+y_0 = x_0+z_0\,,\qquad\qquad  y_l = y_{l+1}\,,\qquad\qquad  z_l = z_{l+1}\,,  \\  
x_l = x_{l+1}+y_{l+1} =x_{l+1}+z_{l+1}\,.
\end{gathered}  \notag
\end{equation}
The generators $y_n$ and $z_n$ for $n>0$ are redundant, and we write the remaining generators with overbars to avoid confusion between $\Mbar$ and $\calM$. Thus, $\Mbar$ is presented by the generators
$$\xbar_0,\ybar_0,\zbar_0,\xbar_1,\xbar_2,\dots$$
and the relations
\begin{equation}  \label{relnsMbar}
\xbar_0+\ybar_0= \xbar_0+\zbar_0\,, \qquad\qquad \xbar_l= \xbar_{l+1}+\ybar_0= \xbar_{l+1}+\zbar_0 \,.
\end{equation}
We shall see below that $\Mbar$ is a quotient of $\calM$ modulo an o-ideal (Lemma \ref{Mbarisom}). This corresponds to the fact that $(\Ebar,\Cbar)$ is a quotient of $(E,C)$ in the sense of \cite[Construction 6.8]{AG}.

\subsection{Structure of $\calM$}  \label{structureM}

It is convenient to identify the following submonoids $\calA_n$ and $\calM_n$ of $\calM$:
$$\calA_n := \sum_{i=1}^n \Zplus a_i \qquad\text{and}\qquad \calM_n := \sum_{i=0}^n \bigl( \Zplus x_i+ \Zplus y_i+ \Zplus z_i \bigr) +\calA_n  \,,$$
for $n\in \N$. We shall see shortly that the submonoid $\Zplus x_0+ \Zplus y_0+ \Zplus z_0$ of $\calM$ can be identified with the monoid $\calM_0$ defined in \S\ref{3graphmons}.

\begin{lemma}  \label{equalMn}
{\rm(a)} For each $n\in\N$, the monoid $\calM_n$ is generated by $a_1,\dots,a_n$, $x_n$, $y_n$, $z_n$.

{\rm(b)} Let $n\in \Zplus$ and $m,m',i,i',j,j',k_1,k'_1,\dots,k_n,k'_n \in \Zplus$. Then
\begin{equation}  \label{Mneqn}
mx_n+ iy_n+ jz_n+ \sum_{l=1}^n k_la_l= m'x_n+ i'y_n+ j'z_n+ \sum_{l=1}^n k'_la_l
\end{equation}
in $\calM$ if and only if
\begin{equation}  \label{Mnconds}
\begin{aligned}
&(m=m'=0, \;\, i=i', \;\, j=j', \;\, k_l=k'_l \;\, \text{for all} \;\, l) \;\, \text{or}  \\
&(m=m'>0, \;\, i+j=i'+j', \;\, k_l=k'_l \;\, \text{for all} \;\, l)
\end{aligned}
\end{equation}

{\rm(c)} The natural map $\calM_0\rightarrow \calM$ gives an isomorphism of $\calM_0$ onto $\Zplus x_0+ \Zplus y_0+ \Zplus z_0$.
\end{lemma}

\begin{proof} (a) This is clear from the fact that
\begin{align*}
y_l &= y_{l+1}+a_{l+1}  &z_l &= z_{l+1}+a_{l+1}  &x_l &= x_{l+1}+y_{l+1}  
\end{align*}
for all $l=0,\dots,n-1$.

(b) Since $x_n+y_n = x_n+z_n$, we have
$$x_n+ iy_n+ jz_n= x_n+ (i+j)y_n \qquad \text{and} \qquad x_n+ i'y_n+ j'z_n= x_n+ (i'+j')y_n \,. $$
This is all that is needed for the implication \eqref{Mnconds}$\Longrightarrow$\eqref{Mneqn}. 

Conversely, assume that \eqref{Mneqn} holds. In view of the presentation of $\calM$ given in \S\ref{3graphmons}, there exists a homomorphism $f: \calM \rightarrow \Zplus$ such that
$$f(x_l)=1 \qquad\quad \text{and} \qquad\quad f(y_l)= f(z_l)= f(a_l) =0$$
for all $l$. Applying $f$ to \eqref{Mneqn} yields $m=m'$.

Assume first that $m=0$, let $A$ be a free abelian group with a basis $\{\beta,\gamma\} \sqcup \{\alpha_l \mid l\in\N \}$, and enlarge $A$ to a monoid $A\sqcup \{\infty\}$ by adjoining an infinity element to $A$. There exists a homomorphism $g: \calM \rightarrow A\sqcup \{\infty\}$ such that
\begin{align*}
g(x_l) &= \infty,  &g(y_l) &= \beta+ \alpha_1+ \cdots+ \alpha_l \,,  \\
g(a_l) &= -\alpha_l \,,  &g(z_l) &= \gamma+ \alpha_1+ \cdots+ \alpha_l
\end{align*}
for all $l$. Applying $g$ to \eqref{Mneqn} yields
$$i\beta+ j\gamma+ \sum_{l=1}^n (i+j-k_l)\alpha_l = i'\beta+ j'\gamma+ \sum_{l=1}^n (i'+j'-k'_l)\alpha_l \,,$$
from which the first alternative of \eqref{Mnconds} follows.

Now suppose that $m>0$. There exists a homomorphism $h: \calM \rightarrow A$ such that
\begin{align*}
h(a_l) &= -\alpha_l \,,  &h(y_l) &= h(z_l) = \beta+ \alpha_1+ \cdots+ \alpha_l \,,  &h(x_l) &= -l\beta+ \sum_{k=1}^l (k-l-1)\alpha_k
\end{align*}
for all $l$. Applying $h$ to \eqref{Mneqn} yields
\begin{align*}
(i+j-mn)\beta &+ \sum_{l=1}^n ( m(l-n-1) +i+j-k_l)\alpha_l  =  \\
 & (i'+j'-mn)\beta+ \sum_{l=1}^n ( m(l-n-1) +i'+j'-k'_l)\alpha_l \,,
 \end{align*}
from which the second alternative of \eqref{Mnconds} follows.

(c) The homomorphism in question, call it $\eta$, sends the generators $x_0$, $y_0$, $z_0$ of $\calM_0$ to the elements of $\calM$ denoted by the same symbols. Part (b) implies that $\eta$ is injective, and the result follows.
\end{proof}

\begin{remark}  \label{M0notref}
We now identify $\calM_0$ with the submonoid $\Zplus x_0+ \Zplus y_0+ \Zplus z_0$ of $\calM$, and we  observe that $\calM_0$ is not a refinement monoid. Namely, it follows from Lemma \ref{equalMn} that $x_0$, $y_0$, and $z_0$ are distinct irreducible elements in $\calM_0$, and hence the equation $x_0+y_0 = x_0+z_0$ has no refinement in $\calM_0$.
\end{remark}

\begin{corollary}  \label{Mnisom}
For $n\in\N$, there are isomorphisms $(\Zplus)^n \oplus\calM_0 \rightarrow \calM_n$ and $(\Zplus)^n \rightarrow \calA_n$ given by the rules
\begin{align*}
\bigl( (k_1,\dots,k_n),\, mx_0+iy_0+jz_0 \bigr) &\longmapsto mx_n+iy_n+jz_n+ k_1a_1+ \cdots+ k_na_n \\
(k_1,\dots,k_n) & \longmapsto k_1a_1+ \cdots+ k_na_n \,.
\end{align*}
\end{corollary}

\begin{proof} By Lemma \ref{equalMn}, the displayed rules give well defined injective maps from $(\Zplus)^n \oplus\calM_0$ to $\calM_n$ and $(\Zplus)^n$ to $\calA_n$, respectively. It is clear that they are surjective homomorphisms.
\end{proof}

The above results allow us to give a direct proof that $\calM$ has refinement, as follows. We will need the equations
\begin{align}  \label{xyzntok}
y_n &= y_k+ \sum_{l=n+1}^k a_l \,,  &z_n &= z_k+ \sum_{l=n+1}^k a_l \,,  &x_n &= x_k+ (k-n)y_k+ \sum_{l=n+1}^k (l-n-1)a_l \,,
\end{align}
for $k>n\ge0$, which follow directly from the relations \eqref{MECrelns}.

\begin{proposition}  \label{Mref}
$\calM$ is a refinement monoid.
\end{proposition}

\begin{proof} Let $b_1,b_2,c_1,c_2 \in \calM$ such that $b_1+b_2= c_1+c_2$, and choose $n\in\N$ such that these elements all lie in $\calM_n$. Write each
$$b_i= \beta_{i0}x_n+ \beta_{i1}y_n+ \beta_{i2}z_n+ b'_i \qquad\text{and}\qquad c_j= \gamma_{j0}x_n+ \gamma_{j1}y_n+ \gamma_{j2}z_n+ c'_j$$
with coefficients $\beta_{is}, \gamma_{js} \in \Zplus$ and elements $b'_i,c'_j \in \calA_n$. Substituting these expressions into the equation $b_1+b_2= c_1+c_2$ and applying Lemma \ref{equalMn}, we find that  $b'_1+b'_2= c'_1+c'_2$ and
\begin{equation}  \label{cut}
\begin{aligned}
(\beta_{10}x_n+ \beta_{11}y_n+ \beta_{12}z_n) &+ (\beta_{20}x_n+ \beta_{21}y_n+ \beta_{22}z_n)  \\
&= (\gamma_{10}x_n+ \gamma_{11}y_n+ \gamma_{12}z_n) + (\gamma_{20}x_n+ \gamma_{21}y_n+ \gamma_{22}z_n) .
\end{aligned}
\end{equation}
The equation $b'_1+b'_2= c'_1+c'_2$ has refinements in $\calA_n$, because that monoid is isomorphic to $(\Zplus)^n$. Any such refinement, added to a refinement for \eqref{cut}, will yield a refinement for $b_1+b_2= c_1+c_2$. Thus, 
\begin{equation}  \label{reducerefinement}
\text{If} \;\, \eqref{cut} \;\, \text{has a refinement, then} \;\, b_1+b_2= c_1+c_2 \;\, \text{has a refinement}.
\end{equation}
This allows us to assume that $b'_i= c'_j =0$ for all $i,j=1,2$.

In view of Lemma \ref{equalMn}, the submonoids $\Zplus x_n+ \Zplus y_n$ and $\Zplus y_n+ \Zplus z_n$ are both isomorphic to $(\Zplus)^2$. Hence, the desired refinement exists in case the elements $b_1$, $b_2$, $c_1$, $c_2$ either all lie in $\Zplus x_n+ \Zplus y_n$ or all lie in $\Zplus y_n+ \Zplus z_n$.

Now assume that at least one of  $b_1$, $b_2$, $c_1$, $c_2$ is not in  $\Zplus x_n+ \Zplus y_n$ and at least one of these elements is not in $\Zplus y_n+ \Zplus z_n$. If $b_1$ and $b_2$ are both in $\Zplus y_n+ \Zplus z_n$, so is $c_1+c_2$, from which we see that $c_1,c_2 \in \Zplus y_n+ \Zplus z_n$, contradicting our choices. Thus, we may assume that $b_1\notin \Zplus y_n+ \Zplus z_n$, and similarly that $c_1\notin \Zplus y_n+ \Zplus z_n$. It follows that $b_1,c_1 \in \N x_n+ \Zplus y_n$. Now $b_2$ and $c_2$ cannot both be in $\Zplus x_n+ \Zplus y_n$; without loss of generality, $b_2 \notin \Zplus x_n+ \Zplus y_n$. It follows that $b_2\in \Zplus y_n+ \N z_n$. Thus, we may assume that the coefficients $\beta_{is}$, $\gamma_{js}$ have been chosen so that
$$\beta_{10}, \gamma_{10}, \beta_{22} > 0, \qquad\quad \beta_{12}, \gamma_{12}, \beta_{20} = 0, \qquad\quad \gamma_{22}=0 \;\, \text{if} \;\, \gamma_{20}>0.$$
There are now two cases to consider, depending on whether $\gamma_{20}$ is zero or not.

First, suppose that $\gamma_{20}=0$. Then
$$\beta_{10}x_n+ (\beta_{11}+ \beta_{21}+ \beta_{22})y_n = \gamma_{10}x_n+ (\gamma_{11}+ \gamma_{21}+ \gamma_{22})y_n \,,$$
and so Lemma \ref{equalMn} implies that $\beta_{10}= \gamma_{10}$ and $\beta_{11}+ \beta_{21}+ \beta_{22} = \gamma_{11}+ \gamma_{21}+ \gamma_{22}$. For any $k>n$, \eqref{xyzntok} shows that
\begin{align*}
b_1 &= \beta_{10}x_k+ (\beta_{10}(k-n)+ \beta_{11})y_k+ \sum_{l=n+1}^k (\beta_{10}(l-n-1)+ \beta_{11})a_l  \\
b_2 &= \beta_{21}y_k+ \beta_{22}z_k+ \sum_{l=n+1}^k (\beta_{21}+\beta_{22})a_l  \\
c_1 &= \gamma_{10}x_k+ (\gamma_{10}(k-n)+ \gamma_{11})y_k+ \sum_{l=n+1}^k (\gamma_{10}(l-n-1)+ \gamma_{11})a_l  \\
c_2 &= \gamma_{21}y_k+ \gamma_{22}z_k+ \sum_{l=n+1}^k (\gamma_{21}+\gamma_{22})a_l  \,.
\end{align*}
Substitute these expressions in the equation $b_1+b_2= c_1+c_2$, and apply \eqref{reducerefinement}. This shows that it suffices to find a refinement for
\begin{align*}
\bigl[ \beta_{10}x_k+ (\beta_{10}(k-n)+ \beta_{11})y_k \bigr] &+ \bigl[ \beta_{21}y_k+ \beta_{22}z_k \bigr] =  \\
 &\bigl[ \gamma_{10}x_k+ (\gamma_{10}(k-n)+ \gamma_{11})y_k \bigr] + \bigl[ \gamma_{21}y_k+ \gamma_{22}z_k \bigr].
\end{align*}
Consequently, we may replace $\beta_{11}$ and $\gamma_{11}$ by $\beta_{10}(k-n)+ \beta_{11}$ and $\gamma_{10}(k-n)+ \gamma_{11}$, for any $k>n$. Therefore, after choosing a sufficiently large $k$ and making the above replacement, we may now assume that
$$\beta_{11} \ge \gamma_{21}+ \gamma_{22} \,.$$

Recall that $\beta_{11}+ \beta_{21}+ \beta_{22} = \gamma_{11}+ \gamma_{21}+ \gamma_{22}$, so that $\beta_{11}- \gamma_{21}- \gamma_{22}= \gamma_{11}- \beta_{21}- \beta_{22}$. Since also $\beta_{10}= \gamma_{10}$, we now have a refinement
$$\bordermatrix{
 &c_1&c_2  \cr
b_1 & \beta_{10}x_n+ (\beta_{11}-\gamma_{21}-\gamma_{22})y_n &\gamma_{21}y_n+ \gamma_{22}z_n  \cr
b_2 &\beta_{21}y_n+ \beta_{22}z_n &0 \cr
}.$$

Finally, suppose that $\gamma_{20} >0$. Then
$$\beta_{10}x_n+ (\beta_{11}+ \beta_{21}+ \beta_{22})y_n = (\gamma_{10}+ \gamma_{20})x_n + (\gamma_{11}+\gamma_{21})y_n \,,$$
whence $\beta_{10}= \gamma_{10}+ \gamma_{20}$ and $\beta_{11}+ \beta_{21}+ \beta_{22}= \gamma_{11}+\gamma_{21}$. As in the previous case, we may assume that $\beta_{11} \ge \gamma_{21}$. This allows the refinement
$$\bordermatrix{
 &c_1&c_2  \cr
b_1 & \gamma_{10}x_n+ (\beta_{11}-\gamma_{21})y_n &\gamma_{20}x_n+ \gamma_{21}y_n  \cr
b_2 &\beta_{21}y_n+ \beta_{22}z_n &0 \cr
},$$
completing the proof.
\end{proof}

\begin{lemma}  \label{stateM}
$\calM$ is conical, stably finite, archimedean, and antisymmetric.
\end{lemma}

\begin{proof} Conicality was noted in \S\ref{3graphmons}, but will also be an immediate consequence of the present proof. Antisymmetry will follow once we have shown that $\calM$ is conical and stably finite.

There is a homomorphism $s: \calM\rightarrow \Q^+$ such that $s(x_n)= s(y_n)= s(z_n)= s(a_n)= 1/2^n$ for all $n$. Since the generators of $\calM$ are all mapped to nonzero elements of $\Q^+$, we have $s^{-1}(0)= \{0\}$. Conicality, stable finiteness, and the archimedean property follow immediately.
\end{proof}
 
 \begin{theorem}  \label{Mwild}
 The monoid $\calM$ is a wild refinement monoid.
 \end{theorem}
 
 \begin{proof} Refinement holds by Proposition \ref{Mref}, and $\calM$ is stably finite by Lemma \ref{stateM}. However, $\calM$ is not cancellative, since $x_0+y_0= x_0+z_0$, whereas $y_0\ne z_0$ by Lemma \ref{equalMn}. Therefore Theorem \ref{tame:stablyfinite} shows that $\calM$ is wild. 
  \end{proof}

\begin{lemma}  \label{Mseparative}
$\calM$ is separative and unperforated.
\end{lemma}

\begin{proof} Since $\calM$ is antisymmetric (Lemma \ref{stateM}), it will follow from lack of perforation that $\calM$ is torsionfree, i.e., $(ma=mb \implies a=b)$ for any $m\in \N$ and $a,b\in \calM$. Torsionfreeness implies separativity. Thus, we just need to prove that $\calM$ is unperforated.

It suffices to prove that each $\calM_n$ is unperforated. Since $(\Zplus)^n$ is unperforated, Corollary \ref{Mnisom} shows that it is enough to prove that $\calM_0$ is unperforated.

Suppose that $m\in\N$ and $a,b\in \calM_0$ with $ma\le mb$ in $\calM_0$, that is, $ma+c =mb$ for some $c\in \calM_0$. Write
\begin{align*}
a &= \alpha_0x_0+ \alpha_1y_0+ \alpha_2z_0 \,,  &b &= \beta_0x_0+ \beta_1y_0+ \beta_2z_0 \,,  &c &= \gamma_0x_0+ \gamma_1y_0+ \gamma_2z_0 \,,
\end{align*}
with coefficients $\alpha_i,\beta_i,\gamma_i \in\Zplus$. Then
$$(m\alpha_0+\gamma_0)x_0+ (m\alpha_1+\gamma_1)y_0+ (m\alpha_2+\gamma_2)z_0 = m\beta_0x_0+ m\beta_1y_0+ m\beta_2z_0 \,,$$
and so Lemma \ref{equalMn} implies that $m\alpha_0+\gamma_0 = m\beta_0$. Thus, $\alpha_0\le \beta_0$.

Assume first that $\beta_0=0$, which forces $\alpha_0= \gamma_0 =0$. By Lemma \ref{equalMn}, $m\alpha_i+ \gamma_i= m\beta_i$ for $i=1,2$, and so each $\alpha_i \le \beta_i$. Consequently,
$$a+ (\beta_1-\alpha_1)y_0+ (\beta_2-\alpha_2)z_0= \beta_1y_0+ \beta_2z_0= b,$$
proving that $a\le b$ in $\calM_0$.

Now assume that $\beta_0>0$. In this case, Lemma \ref{equalMn} implies that 
$$m\alpha_1+\gamma_1+ m\alpha_2+\gamma_2 = m\beta_1+ m\beta_2 \,,$$
and so $\alpha_1+\alpha_2 \le \beta_1+ \beta_2$. Consequently,
\begin{align*}
a+ (\beta_0-\alpha_0)x_0+ (\beta_1+ \beta_2- \alpha_1- \alpha_2)z_0 &= \beta_0x_0+ \alpha_1y_0+ (\beta_1+ \beta_2- \alpha_1)z_0  \\
 &= \beta_0x_0+ (\beta_1+ \beta_2)y_0 =b,
 \end{align*}
and again $a\le b$ in $\calM_0$.
\end{proof}

We conclude the subsection with the following information about the structure of $\calM$. More about the ideals of $\calM$ will appear in the following subsection.

Recall that $u= x_0+y_0$.

\begin{lemma}  \label{Mirredelts}
{\rm (a)} The elements $a_1,a_2,\dots$ are distinct irreducible elements in $\calM$. 

{\rm (b)} The submonoid $J_1 := \sum_{n=1}^\infty \Zplus a_n$
is an o-ideal of $\calM$, and $J_1 \cong (\Zplus)^{(\N)}$.

{\rm (c)} Every nonzero element of $\calM$ dominates at least one $a_n$.

{\rm (d)} $J_1= \ped (\calM )$.

{\rm (e)} For $n\in\N$, we have $na_n\le u$ but $(n+1)a_n \nleq u$.
\end{lemma}

\begin{proof} (a) The $a_n$ are distinct by Lemma \ref{equalMn}. 

Let $A$ be a free abelian group with a basis $\{\alpha_n \mid n\in\N\}$. There is a homomorphism $g: \calM \rightarrow A\sqcup \{\infty\}$ such that
$$g(x_n)= g(y_n)= g(z_n)= \infty \qquad\quad \text{and} \qquad\quad g(a_n)= \alpha_n$$
for all $n$, and since $g$ maps the generators of $\calM$ to nonzero elements of the conical monoid $A\sqcup \{\infty\}$, we see that $g^{-1}(0)= \{0\}$. If $v,w\in \calM$ with $v+w=a_n$ for some $n$, then $g(v)+g(w)= \alpha_n$. Since $\alpha_n$ is irreducible in $A\sqcup \{\infty\}$, either $g(v)=0$ or $g(w)=0$, and thus $v=0$ or $w=0$. This shows that $a_n$ is irreducible in $\calM$.

(b) These properties follow from Proposition \ref{irredelementgenideal}.

(c) Let $b$ be a nonzero element of $\calM$. Then $b\in \calM_n$ for some $n$, and so
$$b= mx_n+ iy_n+ jz_n+ \sum_{l=1}^n k_la_l$$
with coefficients in $\Zplus$, not all zero. If some $k_l>0$, then $b\ge a_l$. Otherwise,
\begin{align*}
b &= mx_n+ iy_n+ jz_n = m(x_{n+1}+y_{n+1}) + i(y_{n+1}+a_{n+1})+ j(z_{n+1}+a_{n+1})  \\
 &= mx_{n+2}+ (2m+i)y_{n+2}+ jz_{n+2}+ (i+j)a_{n+1}+ (m+i+j)a_{n+2} \,,
\end{align*}
from which it follows that $b\ge a_{n+2}$.

(d) Clearly $J_1\subseteq \ped (\calM)$. If $a$ is an irreducible element in $\calM$, then $a_n\le a$ for some $n$ by (c), so that $a=a_n$.
This shows that $\ped (\calM ) \subseteq J_1$. 

(e) It follows from \eqref{xyzntok} that $u= x_n+ (n+1)y_n+ \sum_{l=1}^n la_l$ for $n\in \N$, whence $u\ge na_n$. Suppose $(n+1)a_n \le u$. Then $(n+1)a_n+b =u$ for some $b\in \calM$, say $b\in \calM_k$ for some $k\ge n$. Write $b= mx_k+ iy_k+ jz_k+ \sum_{l=1}^k k_la_l$ with coefficients in $\Zplus$. Since $u= x_k+ (k+1)y_k+ \sum_{l=1}^k la_l$, Lemma \ref{equalMn} implies that $n+1+k_n= n$, which is impossible. Therefore $(n+1)a_n \nleq u$.
\end{proof}

\subsection{Structure of $\Mbar$}

\begin{lemma}  \label{Mbarisom}
There is a surjective homomorphism $q: \calM \rightarrow \Mbar$ such that
\begin{align*}
q(a_n) &= 0,  &q(x_n) &= \xbar_n \,,  &q(y_n) &= \ybar_0 \,, &q(z_n) &= \zbar_0
\end{align*}
for all $n$, and $q$ induces an isomorphism of $\calM/J_1$ onto $\Mbar$.
\end{lemma}

\begin{proof} Since $\ybar_n= \ybar_0$ and $\zbar_n= \zbar_0$ for all $n$, the existence of $q$ is clear. This homomorphism sends all elements of $J_1$ to $0$, and so it induces a homomorphism $\qbar: \calM/J_1\rightarrow \Mbar$. Since $y_n \equiv_{J_1} y_{n+1}$ and $z_n \equiv_{J_1} z_{n+1}$ for all $n$, there is a homomorphism $q': \Mbar \rightarrow \calM/J_1$ such that
\begin{align*}
q'(\ybar_0) &= (y_0/{\equiv}_{J_1}),  &q'(\zbar_0) &= (z_0/{\equiv}_{J_1}),  &q'(\xbar_n) &= (x_n/{\equiv}_{J_1})
\end{align*}
for all $n$. Clearly, $q'$ and $\qbar$ are mutual inverses.
\end{proof}

\begin{corollary}  \label{Mbarunperf}
$\Mbar$ is separative and unperforated.
\end{corollary}

\begin{proof} Separativity and unperforation, which hold in $\calM$ by Lemma \ref{Mseparative}, pass to $\calM/J_1 \cong \Mbar$.
\end{proof}

\begin{lemma}  \label{Mbarstabfin}
$\Mbar$ is conical, stably finite, and antisymmetric.
\end{lemma}

\begin{proof} Set $B := (\Zplus\times \{0\})+ (\Z\times\N)$, which is a conical submonoid of the group $\Z^2$. There is a homomorphism $t: \Mbar\rightarrow B$ such that
$$t(\ybar_0)= t(\zbar_0)= (1,0) \qquad\quad \text{and} \qquad\quad t(\xbar_n)= (1-n,1)$$
for all $n$. Since $t$ sends the generators of $\Mbar$ to nonzero elements of $B$, we have $t^{-1}(0)= \{0\}$. Conicality and stable finiteness follow, and these two properties imply antisymmetry.
\end{proof}

\begin{lemma}  \label{equalMbar}
Let $n \in \Zplus$.

{\rm(a)} $\Zplus \ybar_0+ \Zplus\zbar_0+ \sum_{l=0}^n \Zplus\xbar_l = \Zplus \ybar_0+ \Zplus\zbar_0+ \Zplus\xbar_n$.

{\rm(b)} Let $i,i',j,j',k,k' \in \Zplus$. Then
\begin{equation}  \label{Mbareqn}
i\ybar_0+ j\zbar_0+ k\xbar_n = i'\ybar_0+ j'\zbar_0+ k'\xbar_n
\end{equation}
in $\Mbar$ if and only if
\begin{equation}  \label{Mbarconds}
(k=k'=0, \;\, i=i', \;\, j=j') \quad \text{or} \quad (k=k'>0, \;\, i+j= i'+j').
\end{equation}
\end{lemma}

\begin{proof} (a) This is clear from the fact that $\xbar_l= \xbar_{l+1}+ \ybar_0$ for all $l=0,\dots,n-1$.

(b) Since $i\ybar_0+ j\zbar_0+ \xbar_n= (i+j)\ybar_0+ \xbar_n$ and $i'\ybar_0+ j'\zbar_0+ \xbar_n= (i'+j')\ybar_0+ \xbar_n$, the implication \eqref{Mbarconds}$\Longrightarrow$\eqref{Mbareqn} is clear.

Conversely, assume that \eqref{Mbareqn} holds. There is a homomorphism $f: \Mbar\rightarrow \Zplus$ such that
$$f(\ybar_0)= f(\zbar_0) = 0  \qquad\quad \text{and} \qquad\quad f(\xbar_l) =1$$
for all $l$. Applying $f$ to \eqref{Mbareqn} yields $k=k'$.

Assume first that $k=0$. There is a homomorphism $g: \Mbar \rightarrow (\Zplus)^2\sqcup\{\infty\}$ such that
\begin{align*}
g(\ybar_0) &= (1,0),  &g(\zbar_0) &= (0,1),  &g(\xbar_l) &= \infty
\end{align*}
for all $l$. Applying $g$ to \eqref{Mbareqn} yields $(i,j)= (i',j')$, so $i=i'$ and $j=j'$.

Now suppose that $k>0$, and apply the homomorphism $t$ from the proof of Lemma \ref{Mbarstabfin} to \eqref{Mbareqn}, to get
$$(i+j+k(1-n),\, k)= (i'+j'+k(1-n),\, k).$$
Therefore $i+j= i'+j'$ in this case.
\end{proof} 

\begin{theorem}  \label{Mbarwild}
The monoid $\Mbar$ is a wild refinement monoid.
\end{theorem}

\begin{proof} Since $\Mbar \cong \calM/J_1$, refinement passes from $\calM$ to $\Mbar$. Now $\Mbar$ is stably finite by Lemma \ref{Mbarstabfin}, but it is not cancellative, because $\xbar_0+\ybar_0= \xbar_0+\zbar_0$ while $\ybar_0 \ne \zbar_0$ by Lemma \ref{equalMbar}. Therefore Theorem \ref{tame:stablyfinite} shows that $\Mbar$ is wild.
\end{proof}

\begin{lemma}  \label{Mbarirredelts}
{\rm(a)} The elements $\ybar_0$ and $\zbar_0$ are distinct irreducible elements in $\Mbar$.

{\rm(b)} The submonoid $\Jbar_2 := \Zplus\ybar_0+ \Zplus\zbar_0$ is an o-ideal of $\Mbar$, and $\Jbar_2 \cong (\Zplus)^2$.

{\rm(c)} $\Mbar/\Jbar_2 \cong \Zplus$.

{\rm(d)} Every nonzero element of $\Mbar$ dominates $\ybar_0$ or $\zbar_0$.

{\rm(e)} $\Jbar_2= \ped (\Mbar )$. 

{\rm(f)} $n\ybar_0+ n\zbar_0 \le \xbar_0$ for all $n\in\N$.
\end{lemma}

\begin{proof} (a) As already noted, Lemma \ref{equalMbar} implies that $\ybar_0 \ne \zbar_0$. Recall the homomorphism $g: \Mbar\rightarrow (\Zplus)^2\sqcup\{\infty\}$ in the proof of that lemma. Since $g(\ybar_0)= (1,0)$ and $g(\zbar_0)= (0,1)$ are irreducible elements of $(\Zplus)^2\sqcup\{\infty\}$, it follows that $\ybar_0$ and $\zbar_0$ are irreducible elements in $\Mbar$.

(b) This follows from Proposition \ref{irredelementgenideal}.

(c) Recall the homomorphism $f: \Mbar\rightarrow \Zplus$ from the proof of Lemma \ref{equalMbar}. Since $f$ sends all elements of $\Jbar_2$ to $0$, it induces a homomorphism $\fbar: \Mbar/\Jbar_2 \rightarrow \Zplus$. The homomorphism $\Zplus\rightarrow \Mbar/\Jbar_2$ that sends $1$ to $\xbar_0/{\equiv}_{\Jbar_2}$ is an inverse for $\fbar$.

(d) Any nonzero element $b\in \Mbar$ can be written as $i\ybar_0+ j\zbar_0+ k\xbar_n$ for some $n\in \N$ and $i,j,k\in \Zplus$, not all zero. If $i>0$ or $j>0$, then $b\ge \ybar_0$ or $b\ge \zbar_0$. Otherwise, $b\ge \xbar_n = \xbar_{n+1}+ \ybar_0 \ge \ybar_0$.

(e) Same as in Lemma \ref{Mirredelts}(d). 

(f) It follows from \eqref{relnsMbar} that $\xbar_0= \xbar_{2n}+ n\ybar_0+ n\zbar_0$ for all $n$.
\end{proof}

\subsection{Summary}  \label{summary}
We summarize the main properties of $\calM$ and $\Mbar$.

\begin{theorem}  \label{MMbarsummary}
The monoids $\calM$ and $\Mbar$ are wild refinement monoids. They are conical, stably finite, antisymmetric, separative, and unperforated. The monoid $\calM$ is archimedean, but $\Mbar$ is not. There are o-ideals $J_1\subset J_2$ in $\calM$ such that
\begin{align*}
J_1 &\cong (\Zplus)^{(\N)},  &J_2/J_1 &\cong (\Zplus)^2,  &\calM/J_2 &\cong \Zplus.
\end{align*}
Moreover, $\Mbar\cong \calM/J_1$, and $\Mbar$ has an o-ideal $\Jbar_2$ such that
$\Jbar_2 \cong (\Zplus)^2$ and $\Mbar/\Jbar_2 \cong \Zplus$.
\end{theorem} 

\begin{proof} The properties stated in the first two sentences are proved in Theorems \ref{Mwild}, \ref{Mbarwild}, Lemmas \ref{stateM}, \ref{Mseparative}, \ref{Mbarstabfin}, and Corollary \ref{Mbarunperf}. Further, Lemma \ref{stateM} shows that $\calM$ is archimedean. By Lemma \ref{Mbarirredelts}, $\ybar_0\ne 0$ and $n\ybar_0 \le \xbar_0$ for all $n\in\N$, so $\Mbar$ is not archimedean. The statements about o-ideals and quotients follow from Lemmas \ref{Mirredelts}, \ref{Mbarisom}, \ref{Mbarirredelts}.
\end{proof}

\begin{remark}  \label{MMbarexamples}
The monoids $\calM$ and $\Mbar$ illustrate other aspects of refinement monoid behavior as well. For instance, the archimedean refinement monoid $\calM$ has a non-archimedean quotient $\calM/J_1$. We also note that $\Zplus\, \zbar_0$ 
is an o-ideal of $\Mbar$ and that $\Mbar/\Zplus\, \zbar_0$ has a presentation $\langle x,y \mid x+y=x\rangle$. This gives an example of a non-stably-finite quotient of a stably finite refinement monoid. 
\end{remark}

\section{Open Problems}

\subsection{} Find axioms that characterize tameness for refinement monoids. A theorem on the model of Theorem \ref{tamecriterion:stablyfinite} would be the ideal result, where the characterizing axioms might include 
unperforation and separativity (recall Theorem \ref{tame:sepunperf}). 

\subsection{} Find conditions (C), applying to refinement monoids $M$ with o-ideals $J$, such that tameness of $J$ and $M/J$, together with (C), implies that $M$ is tame. I.e., find conditions under which the converse of  Proposition \ref{prop:tame-closed-under-quotients} holds. This converse fails in general, as Theorem \ref{MMbarsummary} shows.

\section*{Acknowledgments} We thank the referee for his/her thorough reading of the manuscript and for useful suggestions, particularly Proposition \ref{prop:Riesz-prop}, and we thank E. Pardo and F. Wehrung for helpful correspondence and references.


\end{document}